\documentclass{article}
\usepackage{geometry, enumerate, amsthm, amsmath, amssymb, amscd, graphicx, color, url}
\usepackage{helvet}
\usepackage{courier}
\usepackage{type1cm}
\usepackage[bottom]{footmisc}

\numberwithin{equation}{section}

\newtheorem{Prop}[equation]{Proposition}
\newtheorem{Thm}[equation]{Theorem}
\newtheorem{Lem}[equation]{Lemma}
\newtheorem{Cor}[equation]{Corollary}
\theoremstyle{definition}\newtheorem{Def}[equation]{Definition}
\newtheorem{Ex}[equation]{Example}
\newtheorem{Rem}[equation]{Remark}
\theoremstyle{definition}\newtheorem{Ques}[equation]{Question}

\newcommand{\Q}{\mathbb{Q}}
\newcommand{\Z}{\mathbb{Z}}
\newcommand{\Int}{\textnormal{Int}}
\newcommand{\bfi}{\mathbf{i}}
\newcommand{\bfj}{\mathbf{j}}
\newcommand{\bfk}{\mathbf{k}}
\newcommand{\mcE}{\mathcal{E}}
\newcommand{\mcP}{\mathcal{P}}
\newcommand{\mcan}{\mathcal{A}_n}
\newcommand{\Ln}{\Lambda_n}
\newcommand{\Linf}{\Lambda_\infty}
\newcommand{\olK}{\overline{K}}
\newcommand{\subsetUP}{\rotatebox{90}{$\subseteq$}}

\title{Integral closure of rings of integer-valued polynomials on algebras}

\date{\today}

\author{Giulio Peruginelli\footnote{Institut f\"ur Analysis und Comput. Number Theory, Technische Univ., Steyrergasse 30, A-8010 Graz, Austria. E-mail: peruginelli@math.tugraz.at.}\\
Nicholas J. Werner\footnote{Department of Mathematics, The Ohio State University at Newark, Ohio, USA. Email: nwerner@newark.osu.edu.}}

\begin{document}


%

\maketitle

\begin{abstract}
\noindent 
Let $D$ be an integrally closed domain with quotient field $K$. Let $A$ be a torsion-free $D$-algebra that is finitely generated as a $D$-module. For every $a$ in $A$ we consider its minimal polynomial $\mu_a(X)\in D[X]$, i.e. the monic polynomial of least degree such that $\mu_a(a)=0$. The ring $\Int_K(A)$ consists of polynomials in $K[X]$ that send elements of $A$ back to $A$ under evaluation. 
If $D$ has finite residue rings, we show that the integral closure of $\Int_K(A)$ is the ring of polynomials in $K[X]$ which map the roots in an algebraic closure of $K$ of all the $\mu_a(X)$, $a\in A$, into elements that are integral over $D$. The result is obtained by identifying $A$ with a $D$-subalgebra of the matrix algebra $M_n(K)$ for some $n$ and then considering polynomials which map a matrix to a matrix integral over $D$. We also obtain information about polynomially dense subsets of these rings of polynomials. 

\vspace{0.5cm}

\end{abstract}

\section{Introduction}

Let $D$ be a (commutative) integral domain with quotient field $K$. The ring $\Int(D)$ of integer-valued polynomials on $D$ consists of polynomials in $K[X]$ that map elements of $D$ back to $D$. More generally, if $E \subseteq K$, then one may define the ring $\Int(E, D)$ of polynomials that map elements of $E$ into $D$. 

One focus of recent research (\cite{EvFaJoh}, \cite{Frisch}, \cite{LopWer}, \cite{Per1}, \cite{Per2}) has been to generalize the notion of integer-valued polynomial to $D$-algebras. When $A \supseteq D$ is a torsion-free module finite $D$-algebra we define $\Int_K(A) := \{f \in K[X] \mid f(A) \subseteq A\}$. The set $\Int_K(A)$ forms a commutative ring. 
If we assume that $K \cap A = D$, then $\Int_K(A)$ is contained  in $\Int(D)$ (these two facts are indeed equivalent),
and often $\Int_K(A)$ shares properties similar to those of $\Int(D)$ (see the references above, especially \cite{Frisch}). 

When $A = M_n(D)$, the ring of $n \times n$ matrices with entries in $D$, $\Int_K(A)$ has proven to be particularly amenable to investigation. For instance, \cite[Thm.\ 4.6]{LopWer} shows that the integral closure of $\Int_\Q(M_n(\Z))$ is $\Int_\Q(\mcan)$, where $\mcan$ is the set of algebraic integers of degree at most $n$, and $\Int_\Q(\mcan) = \{ f \in \Q[X] \mid f(\mcan) \subseteq \mcan\}$. In this paper, we will generalize this theorem and describe the integral closure of $\Int_K(A)$ when $D$ is an integrally closed domain with finite residue rings. Our description (Theorem \ref{Integral closure of Int(A)}) may be considered an extension of both \cite[Thm.\ 4.6]{LopWer} and \cite[Thm.\ IV.4.7]{CaCh} (the latter originally proved in \cite[Prop. 2.2]{GHL}), which states that if $D$ is Noetherian and $D'$ is its integral closure in $K$,  then the integral closure of $\Int(D)$ equals $\Int(D,D')=\{f\in K[X] \mid f(D)\subset D'\}$.

Key to our work will be rings of polynomials that we dub \textit{integral-valued polynomials}, and which act on certain subsets of $M_n(K)$. Let $\olK$ be an algebraic closure of $K$.  We will establish a close connection between integral-valued polynomials and polynomials that act on elements of $\olK$ that are integral over $D$. We will also investigate polynomially dense subsets of rings of integral-valued polynomials.

In Section \ref{Integral-valued}, we define what we mean by integral-valued polynomials, discuss when sets of such polynomials form a ring, and connect them to the integral elements of $\olK$. In Section \ref{Algebras}, we apply the results of Section \ref{Integral-valued} to $\Int_K(A)$ and prove the aforementioned theorem about its integral closure. Section \ref{Matrix rings} covers polynomially dense subsets of rings of integral-valued polynomials. We close by posing several problems for further research.

\section{Integral-valued polynomials}\label{Integral-valued}

Throughout, we assume that $D$ is an integrally closed domain with quotient field $K$. We denote by $\olK$ a fixed algebraic closure of $K$. When working in $M_n(K)$, we associate $K$ with the scalar matrices, so that we may consider $K$ (and $D$) to be subsets of $M_n(K)$.

For each matrix $M \in M_n(K)$, we let $\mu_M(X) \in K[X]$ denote the minimal polynomial of $M$, which is the monic generator of $N_{K[X]}(M) = \{f \in K[X] \mid f(M) = 0\}$, called the null ideal of $M$. We define $\Omega_M$ to be the set of eigenvalues of $M$ considered as a matrix in $M_n(\olK)$, which are the roots of $\mu_M$ in $\olK$. For a subset $S \subseteq M_n(K)$, we define $\Omega_S := \bigcup_{M \in S} \Omega_M$. Note that a matrix in $M_n(K)$ may have minimal polynomial in $D[X]$ even though the matrix itself is not in $M_n(D)$. A simple example is given by 
$\big(\begin{smallmatrix}
0 & q\\ 0 & 0
\end{smallmatrix}\big)
\in M_2(K)$, where $q \in K \setminus D$.

\begin{Def}\label{Integral}
We say that $M \in M_n(K)$ is \textit{integral over D} (or just \textit{integral}, or is an \textit{integral matrix}) if $M$ solves a monic polynomial in $D[X]$. A subset $S$ of $M_n(K)$ is said to be \textit{integral} if each $M \in S$ is integral over $D$.
\end{Def}

Our first lemma gives equivalent definitions for a matrix to be integral.

\begin{Lem}\label{Integral conditions}
Let $M \in M_n(K)$. The following are equivalent:
\begin{enumerate}[(i)]
\item $M$ is integral over $D$
\item $\mu_M \in D[X]$
\item each $\alpha \in \Omega_M$ is integral over $D$
\end{enumerate}
\end{Lem}
\begin{proof}
$(i) \Rightarrow (iii)$ 
Suppose $M$ solves a monic polynomial $f(X)$ with coefficients in $D$. As $\mu_M(X)$ divides $f(X)$, its roots are then also roots of  $f(X)$. Hence, the elements of $\Omega_M$ are integral over $D$.

$(iii) \Rightarrow (ii)$ The coefficients of $\mu_M \in K[X]$ are the elementary symmetric functions of its roots. Assuming $(iii)$ holds, these roots are integral over $D$, hence the coefficients of $\mu_M$ are integral over $D$. Since $D$ is integrally closed, we must have $\mu_M \in D[X]$.

$(ii) \Rightarrow (i)$ Obvious.
%
%
\end{proof}

For the rest of this section, we will study polynomials in $K[X]$ that take values on sets of integral matrices. These are the integral-valued polynomials mentioned in the introduction.

\begin{Def}\label{Integral-valued polynomials}
Let $S \subseteq M_n(K)$. Let $K[S]$ denote the $K$-subalgebra of $M_n(K)$ generated by $K$ and the elements of $S$. Define $S' := \{M \in K[S] \mid M \text{ is integral}\}$ and $\Int_K(S, S') := \{f \in K[X] \mid f(S) \subseteq S'\}$. We call $\Int_K(S, S')$ a set of \textit{integral-valued polynomials}.
\end{Def}

\begin{Rem}
In the next lemma, we will prove that forming the set $S'$ is a closure operation in the sense that $(S')' = S'$. We point out that this construction differs from the usual notion of integral closure in several ways. First, if $S$ itself is not integral, then $S \not\subseteq S'$. Second, $S'$ need not have a ring structure. Indeed, if $D = \Z$ and $S = M_2(\Z)$, then both 
$\big(\begin{smallmatrix}
1 & 0\\ 0 & 0
\end{smallmatrix}\big)$ 
and 
$\big(\begin{smallmatrix}
1/2 \; & 1/2 \\ 1/2 \; & 1/2
\end{smallmatrix}\big)$ 
are in $S'$, but neither their sum nor their product is integral. Lastly, even if $S$ is a commutative ring then $S'$ need not be the same as the integral closure of $S$ in $K[S]$, because we insist that the elements of $S$ satisfy a monic polynomial in $D[X]$ rather than $S[X]$.

However, if $S$ is a commutative $D$-algebra and it is an integral subset of $M_n(K)$ then $S'$ is equal to the integral closure of $S$ in $K[S]$ (see Corollary 1 to Proposition 2 and Proposition 6 of \cite[Chapt. V]{Bourbaki}).
\end{Rem}

\begin{Lem}\label{(S')' = S'}
Let $S \subseteq M_n(K)$. Then, $(S')' = S'$.
\end{Lem}
\begin{proof}
We just need to show that $K[S'] = K[S]$. By definition, $S' \subseteq K[S]$, so $K[S'] \subseteq K[S]$. For the other containment, let $M \in K[S]$. Let $d \in D$ be a common multiple for all the denominators of the entries in $M$. Then, $d M \in S' \subseteq K[S']$. Since $1/d \in K$, we get $M \in K[S']$.
\end{proof}

An integral subset of $M_n(K)$ need not be closed under addition or multiplication, so at first glance it may not be clear that $\Int_K(S, S')$ is closed under these operations. As we now show, $\Int_K(S, S')$ is in fact a ring.

\begin{Prop}\label{Int(S, olS) is a ring}
Let $S \subseteq M_n(K)$. Then, $\Int_K(S, S')$ is a ring, and if $D \subseteq S$, then $\Int_K(S, S') \subseteq \Int(D)$.
\end{Prop}
\begin{proof}
Let $M \in S$ and $f, g \in \Int_K(S, S')$. Then, $f(M), g(M)$ are integral over $D$. By Corollary 2 after Proposition 4 of \cite[Chap.\ V]{Bourbaki}, the $D$-algebra generated by $f(M)$ and $g(M)$ is integral over $D$. So, $f(M)+g(M)$ and $f(M)g(M)$ are both integral over $D$ and are both in $K[S]$. Thus, $f(M)+g(M), f(M)g(M) \in S'$ and $f+g, fg \in \Int_K(S, S')$. 
Assuming $D \subseteq S$, let $f \in \Int_K(S, S')$ and $d \in D$. Then, $f(d)$ is an integral element of $K$. Since $D$ is integrally closed, $f(d) \in D$. Thus, $\Int_K(S, S') \subseteq \Int(D)$.
\end{proof}

We now begin to connect our rings of integral-valued polynomials to rings of polynomials that act on elements of $\olK$ that are integral over $D$. For each $n > 0$, let 
\begin{equation*}
\Ln := \{\alpha \in \olK \mid \alpha \text{ solves a monic polynomial in } D[X] \text{ of degree } n\}
\end{equation*}
In the special case $D = \Z$, we let $\mcan := \Lambda_n = \{\text{ algebraic integers of degree at most $n$ }\}\subset\overline{\Q}$.

For any subset $\mcE$ of $\Ln$, define 
\begin{equation*}
\Int_K(\mcE, \Ln) := \{ f \in K[X] \mid f(\mcE) \subseteq \Ln\}
\end{equation*}
to be the set of polynomials in $K[X]$ mapping elements of $\mcE$ into $\Ln$. If $\mcE = \Ln$, then we write simply $\Int_K(\Ln)$. As with $\Int_K(S, S')$, $\Int_K(\mcE, \Ln)$ is a ring despite the fact that $\Ln$ is not.

\begin{Prop}\label{Algebraic integers}
For any $\mcE \subseteq \Ln$, $\Int_K(\mcE, \Ln)$ is a ring, and is integrally closed.
\end{Prop}
\begin{proof}
Let $\Linf$ be the integral closure of $D$ in $\olK$. We set $\Int_{\olK}(\mcE, \Linf) = \{f \in \olK[X] \mid f(\mcE) \subseteq \Linf\}$. Then, $\Int_{\olK}(\mcE, \Linf)$ is a ring, and by \cite[Prop.\ IV.4.1]{CaCh} it is integrally closed.

Let $\Int_K(\mcE, \Linf) = \{f \in K[X] \mid f(\mcE) \subseteq \Linf\}$. Clearly, $\Int_K(\mcE, \Ln) \subseteq \Int_K(\mcE, \Linf)$. However, if $\alpha \in \mcE$ and $f \in K[X]$, then $[K(f(\alpha)) : K] \leq [K(\alpha) : K] \leq n$, so in fact $\Int_K(\mcE, \Ln) = \Int_K(\mcE, \Linf)$. Finally, since $\Int_K(\mcE, \Linf) = \Int_{\olK}(\mcE, \Linf) \cap K[X]$ is the contraction of $\Int_{\olK}(\mcE, \Linf)$ to $K[X]$, it is an integrally closed ring, proving the proposition.
\end{proof}

Theorem 4.6 in \cite{LopWer} shows that the integral closure of $\Int_\Q(M_n(\Z))$ equals the ring $\Int_\Q(\mcan)$. As we shall see (Theorem \ref{Equality of int rings}), this is evidence of a broader connection between the rings of integral-valued polynomials $\Int_K(S, S')$ and rings of polynomials that act on elements of $\Ln$. The key to this connection is the observation contained in Lemma \ref{Integral conditions} that the eigenvalues of an integral matrix in $M_n(K)$ lie in $\Ln$ and also the well known fact that if $M \in M_n(K)$ and $f \in K[X]$, then the eigenvalues of $f(M)$ are exactly $f(\alpha)$, where $\alpha$ is an eigenvalue of $M$. More precisely, if $\chi_M(X)=\prod_{i=1,\ldots,n}(X-\alpha_i)$ is the characteristic polynomial of $M$ (the roots $\alpha_i$ are in $\olK$ and there may be repetitions), then the characteristic polynomial of $f(M)$ is $\chi_{f(M)}(X)=\prod_{i=1,\ldots,n}(X-f(\alpha_i))$. Phrased in terms of our $\Omega$-notation, we have:
\begin{equation}\label{Eigenvalues}
\text{if $M \in M_n(K)$ and $f \in K[X]$, then $\Omega_{f(M)} = f(\Omega_M) = \{f(\alpha) \mid \alpha \in \Omega_M \}$}.
\end{equation}

Using this fact and our previous work, we can equate $\Int_K(S, S')$ with a ring of the form $\Int_K(\mcE, \Ln)$.

\begin{Thm}\label{Equality of int rings}
Let $S \subseteq M_n(K)$. Then, $\Int_K(S, S') = \Int_K(\Omega_S, \Ln)$, and in particular $\Int_K(S, S')$ is integrally closed.
\end{Thm}
\begin{proof}
We first prove this for $S = \{M\}$. Using Lemma \ref{Integral conditions} and (\ref{Eigenvalues}), for each $f \in K[X]$ we have:
\begin{equation*}
f(M) \in S' \iff f(M) \text{ is integral } \iff \Omega_{f(M)} \subseteq \Ln \iff f(\Omega_M) \subseteq \Ln.
\end{equation*}
This proves that $\Int_K(\{M\}, \{M\}') = \Int_K(\Omega_{M}, \Ln)$. For a general subset $S$ of $M_n(K)$, we have
\begin{equation*}
\Int_K(S, S') = \bigcap_{M \in S} \Int_K(\{M\}, S') = \bigcap_{M \in S} \Int_K(\Omega_M, \Ln) = \Int_K(\Omega_S, \Ln).
\end{equation*} 
\end{proof}

The above proof shows that if a polynomial is integral-valued on a matrix, then it is also integral-valued on any other matrix with the same set of eigenvalues. 
Note that for a single integral matrix $M$ we have these inclusions:
$$D\subset D[M]\subseteq\{M\}'\subset K[M].$$
Moreover, $\{M\}'$ is equal to the integral closure of $D[M]$ in $K[M]$ (because $D[M]$ is a commutative algebra).

\section{The case of a $D$-algebra}\label{Algebras}

We now use the results from Section \ref{Integral-valued} to gain information about $\Int_K(A)$, where $A$ is a $D$-algebra. In Theorem \ref{Integral closure of Int(A)} below, we shall obtain a description of the integral closure of $\Int_K(A)$.

As mentioned in the introduction, we assume that $A$ is a torsion-free $D$-algebra that is finitely generated as a $D$-module. Let $B = A \otimes_D K$ be the extension of $A$ to a $K$-algebra. Since $A$ is a faithful $D$-module, $B$ contains copies of $D$, $A$, and $K$. Furthermore, $K$ is contained in the center of $B$, so we can evaluate polynomials in $K[X]$ at elements of $B$ and define
\begin{equation*}
\Int_K(A) := \{f \in K[X] \mid f(A) \subseteq A\}.
\end{equation*}

Letting $n$ be the vector space dimension of $B$ over $K$, we also have an embedding $B \hookrightarrow M_n(K)$, $b\mapsto M_b$. More precisely, we may embed $B$ into the ring of $K$-linear endomorphisms of $B$ (which is isomorphic to $M_n(K)$) via the map $B \hookrightarrow \text{End}_K(B)$ sending $b \in B$ to the endomorphism $x \mapsto b\cdot x$. Consequently, starting with just $D$ and $A$, we obtain a representation of $A$ as a $D$-subalgebra of $M_n(K)$. Note that $n$ may be less than the minimum number of generators of $A$ as a $D$-module.

In light of the aforementioned matrix representation of $B$, several of the definitions and notations we defined in Section \ref{Integral-valued} will carry over to $B$. Since the concepts of minimal polynomial and eigenvalue are independent of the representation $B \hookrightarrow M_n(K)$, the following are well-defined:
\begin{itemize}
\item for all $b \in B$, $\mu_b(X) \in K[X]$ is the minimal polynomial of $b$. So, $\mu_b(X)$ is the monic polynomial of minimal degree in $K[X]$ that kills $b$. Equivalently, $\mu_b$ is the monic generator of the null ideal $N_{K[X]}(b)$ of $b$. This is the same as the minimal polynomial of $M_b\in M_n(K)$, since for all $f\in K[X]$ we have $f(M_b)=M_{f(b)}$. To ease the notation, from now on we will identify $b$ with $M_b$.
\item by the Cayley-Hamilton Theorem, $\deg(\mu_b)\leq n$, for all $b\in B$.
\item for all $b \in B$, $\Omega_b = \{\text{roots of } \mu_b \text{ in } \olK\}$. The elements of $\Omega_b$ are nothing else than the eigenvalues of $b$ under any matrix representation $B \hookrightarrow M_n(K)$. If $S \subseteq B$, then $\Omega_S = \bigcup_{b \in S} \Omega_b$.
\item $b \in B$ is \textit{integral over D} (or just \textit{integral}) if $b$ solves a monic polynomial in $D[X]$.
\item $B = K[A]$, since $B$ is formed by extension of scalars from $D$ to $K$.
\item $A' = \{b \in B \mid b \text{ is integral}\}$. 
By \cite[Theorem 1, Chapt. V]{Bourbaki} $A\subseteq A'$. In particular, this implies $A\cap K=D$ (because $D$ is integrally closed), so that $\Int_K(A)\subseteq\Int(D)$ (see the remarks in the introduction).
\item $\Int_K(A, A') = \{f \in K[X] \mid f(A) \subseteq A'\}$.
\end{itemize}

Working exactly as in Proposition \ref{Int(S, olS) is a ring}, we find that $\Int_K(A, A')$ is another ring of integral-valued polynomials. Additionally, Lemma \ref{Integral conditions} and (\ref{Eigenvalues}) hold for elements of $B$. 
Consequently, we have
\begin{Thm}\label{Int_K(A, A') is integrally closed}
$\Int_K(A, A')$ is an integrally closed ring and is equal to $\Int_K(\Omega_A, \Ln)$.
\end{Thm}
By generalizing results from $\cite{LopWer}$, we will show that if $D$ has finite residue rings, then $\Int_K(A, A')$ is the integral closure of $\Int_K(A)$. This establishes the analogue of  \cite[Thm.\ IV.4.7]{CaCh} (originally proved in \cite[Prop. 2.2]{GHL}) mentioned in the introduction.

We will actually prove a slightly stronger statement and give a description of $\Int_K(A, A')$ as the integral closure of an intersection of pullbacks. Notice that
\begin{equation*}
\bigcap_{a \in A} (D[X] + \mu_a(X) \cdot K[X]) \subseteq \Int_K(A)
\end{equation*}
because if $f \in D[X] + \mu_a(X) \cdot K[X]$, then $f(a) \in D[a] \subseteq A$. We thus have a chain of inclusions
\begin{equation}\label{Pullbacks}
\bigcap_{a \in A} (D[X] + \mu_a(X) \cdot K[X]) \subseteq \Int_K(A) \subseteq \Int_K(A, A')
\end{equation}
and our work below will show that this is actually a chain of integral ring extensions.

\begin{Lem}\label{Int cl lemma}
Let $f \in \Int_K(A, A')$, and write $f(X) = g(X)/d$ for some $g \in D[X]$ and some nonzero $d \in D$. Then, for each $h \in D[X]$, $d^{n-1}h(f(X)) \in \bigcap_{a \in A} ( D[X] + \mu_a(X) \cdot K[X])$.
\end{Lem}
\begin{proof}
Let $a \in A$. Since $f \in \Int_K(A, A')$, $m := \mu_{f(a)} \in D[X]$, and $\deg(m) \leq n$.

Now, $m$ is monic, so we can divide $h$ by $m$ to get $h(X) = q(X)m(X) + r(X)$, where $q, r \in D[X]$, and either $r = 0$ or $\deg(r) < n$. Then,
\begin{equation*}
d^{n-1} h(f(X)) = d^{n-1} q(f(X)) m(f(X)) + d^{n-1} r(f(X))
\end{equation*}
The polynomial $d^{n-1} q(f(X)) m(f(X)) \in K[X]$ is divisible by $\mu_a(X)$ because $m(f(a)) = 0$. Since $\deg(r) < n$, $d^{n-1} r(f(X)) \in D[X]$. Thus, $d^{n-1} h(f(X)) \in D[X] + \mu_a(X) \cdot K[X]$, and since $a$ was arbitrary the lemma is true.
\end{proof}

For the next result, we need an additional assumption. Recall that a ring $D$ has finite residue rings if for all proper nonzero ideal $I\subset D$, the residue ring $D/I$ is finite. Clearly, this condition is equivalent to asking that for all nonzero $d\in D$, the residue ring $D/dD$ is finite.

\begin{Thm}\label{Integral closure of Int(A)}
Assume that $D$ has finite residue rings. Then, $\Int_K(A, A') = \Int_K(\Omega_A, \Ln)$ is the integral closure of both $\bigcap_{a \in A} (D[X] + \mu_a(X) \cdot K[X])$ and $\Int_K(A)$.
\end{Thm}
\begin{proof}
Let $R = \bigcap_{a \in A} (D[X] + \mu_a(X) \cdot K[X])$. By (\ref{Pullbacks}), it suffices to prove that $\Int_K(A, A')$ is the integral closure of $R$. Let $f(X) = g(X)/d \in \Int_K(A, A')$. By Theorem \ref{Int_K(A, A') is integrally closed}, $\Int_K(A, A')$ is integrally closed, so it is enough to find a monic polynomial $\phi \in D[X]$ such that $\phi(f(X)) \in R$.

Let $\mcP \subseteq D[X]$ be a set of monic residue representatives for $\{\mu_{f(a)}(X)\}_{a \in A}$ modulo $(d^{n-1})^2$.

Since $D$ has finite residue rings, $\mcP$ is finite. Let $\phi(X)$ be the product of all the polynomials in $\mcP$. Then, $\phi$ is monic and is in $D[X]$.

Fix $a \in A$ and let $m = \mu_{f(a)}$. There exists $p(X) \in \mcP$ such that $p(X)$ is equivalent to $m$ mod $(d^{n-1})^2$, so $p(X) = m(X) + (d^{n-1})^2 r(X)$ for some $r \in D[X]$. Furthermore, $p(X)$ divides $\phi(X)$, so there exists $q(X) \in D[X]$ such that $\phi(X) = p(X) q(X)$. Thus,
\begin{align*}
\phi(f(X)) &= p(f(X)) q(f(X))\\
&= m(f(X))q(f(X)) + (d^{n-1})^2 r(f(X)) q(f(X))\\
&= m(f(X))q(f(X)) + d^{n-1} r(f(X)) \cdot d^{n-1} q(f(X))\\
\end{align*}

As in Lemma \ref{Int cl lemma}, $m(f(X))q(f(X)) \in \mu_a(X) \cdot K[X]$ because $m(f(a)) = 0$. By Lemma \ref{Int cl lemma},
$d^{n-1} r(f(X))$ and  $d^{n-1} q(f(X))$ are in $D[X] + \mu_a(X)\cdot K[X]$. 
Hence, $\phi(f(X)) \in D[X] + \mu_a(X) \cdot K[X]$, and since $a$ was arbitrary, $\phi(f(X)) \in R$.
\end{proof}

Theorem \ref{Integral closure of Int(A)} says that the integral closure of $\Int_K(A)$ is equal to the ring of polynomials in $K[X]$ which map the eigenvalues of all the elements $a\in A$ to integral elements over $D$.

\begin{Rem}\label{Pullback remark}
By following essentially the same steps as in Lemma \ref{Int cl lemma} and Theorem \ref{Integral closure of Int(A)}, one may prove that $\Int_K(A, A')$ is the integral closure of $\Int_K(A)$ without the use of the pullbacks $D[X] + \mu_a(X) \cdot K[X]$. However, employing the pullbacks gives a slightly stronger theorem without any additional difficulty.

In the case $A=M_n(D)$, $\Int_K(M_n(D))$ is equal to the intersection of the pullbacks $D[X] + \mu_M(X) \cdot K[X]$, for $M\in M_n(D)$. Indeed, let $f\in \Int_K(M_n(D))$ and $M\in M_n(D)$. By \cite[Remark 2.1 \& (3)]{Per2}, $\Int_K(M_n(D))$ is equal to the intersection of the pullbacks $D[X]+\chi_M(X)K[X]$, for $M\in M_n(D)$, where $\chi_M(X)$ is the characteristic polynomial of $M$. 
By the Cayley-Hamilton Theorem, $\mu_M(X)$ divides $\chi_M(X)$ so that $f\in D[X]+\chi_M(X)K[X]\subseteq D[X]+\mu_M(X)K[X]$ and we are done.
\end{Rem}

\begin{Rem}
The assumption that $D$ has finite residue rings implies that $D$ is Noetherian. Given that \cite[Thm.\ IV.4.7]{CaCh} (or \cite[Prop. 2.2]{GHL}) requires only the assumption that $D$ is Noetherian, it is fair to ask if Theorem \ref{Integral closure of Int(A)} holds under the weaker condition that $D$ is Noetherian. 
\end{Rem}

Note that $\Omega_{M_n(D)} = \Ln$ (and in particular, $\Omega_{M_n(\Z)} = \mcan$). Hence, we obtain the following (which generalizes \cite[Thm.\ 4.6]{LopWer}):
\begin{Cor}\label{Int cl of M_n(Z)}
If $D$ has finite residue rings, then the integral closure of $\Int_K(M_n(D))$ is $\Int_K(\Ln)$.
\end{Cor}

The algebra of upper triangular matrices yields another interesting example.

\begin{Cor}\label{Triangular matrices}
Assume that $D$ has finite residue rings. For each $n > 0$, let $T_n(D)$ be the ring of $n \times n$ upper triangular matrices with entries in $D$. Then, the integral closure of $\Int_K(T_n(D))$ equals $\Int(D)$.
\end{Cor}
\begin{proof}
For each $a \in T_n(D)$, $\mu_a$ splits completely and has roots in $D$, so $\Omega_{T_n(D)} = D$. Hence, the integral closure of $\Int_K(T_n(D))$ is $\Int_K(\Omega_{T_n(D)}, \Ln) = \Int_K(D, \Ln)$. But, polynomials in $K[X]$ that move $D$ into $\Ln$ actually move $D$ into $\Ln \cap K = D$. Thus, $\Int_K(D, \Ln) = \Int(D)$.
\end{proof}

Since $\Int_K(T_{n}(D)) \subseteq \Int_K(T_{n-1}(D))$ for all $n > 0$, the previous proposition proves that
\begin{equation*}
\cdots \subseteq \Int_K(T_n(D)) \subseteq \Int_K(T_{n-1}(D)) \subseteq \cdots \subseteq \Int(D)
\end{equation*}
is a chain of integral ring extensions.

\section{Matrix rings and polynomially dense subsets}\label{Matrix rings}

For any $D$-algebra $A$, we have $(A')' = A'$, so $\Int_K(A')$ is integrally closed by Theorem \ref{Int_K(A, A') is integrally closed}. Furthermore, $\Int_K(A')$ is always contained in $\Int_K(A, A')$. One may then ask: when does $\Int_K(A')$ equal $\Int_K(A, A')$? In this section, we investigate this question and attempt to identify \textit{polynomially dense} subsets of rings of integral-valued polynomials. The theory presented here is far from complete, so we raise several related questions worthy of future research.

\begin{Def}\label{Polynomially dense}
Let $S \subseteq T \subseteq M_n(K)$. Define $\Int_K(S, T) := \{f \in K[X] \mid f(S) \subseteq T\}$ and $\Int_K(T) := \Int_K(T, T)$. To say that $S$ is \textit{polynomially dense in T} means that $\Int_K(S, T) = \Int_K(T)$.
\end{Def}

Thus, the question posed at the start of this section can be phrased as: is $A$ polynomially dense in $A'$? 

In general, it is not clear how to produce polynomially dense subsets of $A'$, but we can describe some polynomially dense subsets of $M_n(D)'$.

\begin{Prop}\label{Other poly dense sets}
For each $\Omega \subset \Ln$ of cardinality at most $n$, choose $M \in M_n(D)'$ such that $\Omega_M = \Omega$. Let $S$ be the set formed by such matrices. Then, $S$ is polynomially dense in $M_n(D)'$. In particular, the set of companion matrices in $M_n(D)$ is polynomially dense in $M_n(D)'$.
\end{Prop}
\begin{proof}
We know that $\Int_K(M_n(D)') \subseteq \Int_K(S, M_n(D)')$, so we must show that the other containment holds. Let $f \in \Int_K(S, M_n(D)')$ and $N \in M_n(D)'$. Let $M \in S$ such that $\Omega_M = \Omega_N$. Then, $f(M)$ is integral, so by Lemma \ref{Integral conditions} and (\ref{Eigenvalues}), $f(N)$ is also integral.

The proposition holds for the set of companion matrices because for any $\Omega \subset \Ln$, we can find a companion matrix in $M_n(D)$ whose eigenvalues are the elements of $\Omega$.
\end{proof}

By the proposition, any subset of $M_n(D)$ containing the set of companion matrices is polynomially dense in $M_n(D)'$. In particular, $M_n(D)$ is polynomially dense in $M_n(D)'$. 

When $D = \Z$, we can say more. In \cite{Per1} it is shown that $\Int_{\Q}(\mathcal{A}_n)=\Int_{\Q}(A_n,\mathcal{A}_n)$, where $A_n$ is the set of algebraic integers of degree  equal to $n$. Letting $\mathcal{I}$ be the set of companion matrices in $M_n(\Z)$ of irreducible polynomials, we have $\Omega_{\mathcal{I}} = A_n$. Hence, by Corollary \ref{Int cl of M_n(Z)} and Theorem \ref{Equality of int rings}, $\mathcal{I}$ is polynomially dense in $M_n(\Z)'$.

Returning to the case of a general $D$-algebra $A$, the following diagram summarizes the relationships among the various polynomial rings we have considered:
\begin{equation}\label{Diagram}
\begin{array}{cccc}
\Int_K(A) \subseteq & \Int_K(A,A')	&=		& \Int_K(\Omega_A, \Ln)\\
& \subsetUP			&			& \subsetUP \\
& \Int_K(A')		&=		& \Int_K(\Omega_{A'}, \Ln)\\
& \subsetUP			&			& \subsetUP \\
& \Int_K(M_n(D)')	&=		& \Int_K(\Ln)
\end{array}
\end{equation}
From this diagram, we deduce that $A$ is polynomially dense in $A'$ if and only if $\Omega_A$ is polynomially dense in $\Omega_{A'}$.

It is fair to ask what other relationships hold among these rings. We present several examples and a proposition concerning possible equalities in the diagram. Again, we point out that such equalities can be phrased in terms of polynomially dense subsets. First, we show that $\Int_K(A)$ need not equal $\Int_K(A, A')$ (that is, $A$ need not be polynomially dense in $A'$).

\begin{Ex}\label{Z[sqrt-3]}
Take $D = \Z$ and $A = \Z[\sqrt{-3}]$. Then, $A' = \Z[\theta]$, where $\theta = \frac{1 + \sqrt{-3}}{2}$. The ring $\Int_K(A, A')$ contains both $\Int_K(A)$ and $\Int_K(A')$.

If $\Int_K(A, A')$ equaled $\Int_K(A')$, then we would have $\Int_K(A) \subseteq \Int_K(A')$. However, this is not the case. Indeed, working mod 2, we see that for all $\alpha = a + b\sqrt{-3} \in A$, $\alpha^2 \equiv a^2 - 3b^2 \equiv a^2 + b^2$. So, $\alpha^2(\alpha^2 + 1)$ is always divisible by 2, and hence $\frac{x^2(x^2+1)}{2} \in \Int_K(A)$. On the other hand, $\frac{\theta^2(\theta^2 + 1)}{2} = -\frac{1}{2}$, so $\frac{x^2(x^2+1)}{2} \notin \Int_K(A')$. Thus, we conclude that $\Int_K(A') \subsetneqq \Int_K(A, A')$.
\end{Ex}

The work in the previous example suggests the following proposition.

\begin{Prop}\label{CC IV.4.9}
Assume that $D$ has finite residue rings. Then, $A$ is polynomially dense in $A'$ if and only if $\Int_K(A) \subseteq \Int_K(A')$.
\end{Prop}
\begin{proof}
This is similar to \cite[Thm.\ IV.4.9]{CaCh}. If $A$ is polynomially dense in $A'$, then $\Int_K(A') = \Int_K(A, A')$, and we are done because $\Int_K(A, A')$ always contains $\Int_K(A)$. Conversely, assume that $\Int_K(A) \subseteq \Int_K(A')$. Then, $\Int_K(A) \subseteq \Int_K(A') \subseteq \Int_K(A, A')$. Since $\Int_K(A')$ is integrally closed  by Theorem \ref{Int_K(A, A') is integrally closed} and $\Int_K(A, A')$ is integral over $\Int_K(A)$ by Theorem \ref{Integral closure of Int(A)}, we must have $\Int_K(A') = \Int_K(A, A')$.
\end{proof}

By Proposition \ref{Other poly dense sets}, $\Int_K(M_n(D)') = \Int_K(M_n(D), M_n(D)')$. There exist algebras other than matrix rings for which $\Int_K(A') = \Int_K(A, A')$. We now present two such examples.

\begin{Ex}\label{Triangular mats}
Let $A = T_n(D)$, the ring of $n \times n$ upper triangular matrices with entries in $D$. Define $T_n(K)$ similarly. Then, $A'$ consists of the integral matrices in $T_n(K)$, and since $D$ is integrally closed, such matrices must have diagonal entries in $D$. Thus, $\Omega_{A'} = D = \Omega_{A}$. It follows that $\Int_K(T_n(D), T_n(D)') = \Int_K(T_n(D)')$.
\end{Ex}

\begin{Ex}\label{Hurwitz quats}
Let $\bfi$, $\bfj$, and $\bfk$ be the standard quaternion units satisfying $\bfi^2 = \bfj^2 = \bfk^2 = -1$ and $\bfi\bfj = \bfk = -\bfj\bfi$ 
(see e.g.\ \cite[Ex.\ 1.1, 1.13]{Lam} or \cite{Dickson} for basic material on quaternions).

Let $A$ be the $\Z$-algebra consisting of Hurwitz quaternions:
\begin{equation*}
A = \{a_0 + a_1\bfi + a_2\bfj + a_3\bfk \mid a_\ell \in \Z \text{ for all $\ell$ or } a_\ell \in \Z + \tfrac{1}{2} \text{ for all $\ell$} \}
\end{equation*}
Then, for $B$ we have
\begin{equation*}
B = \{q_0 + q_1\bfi + q_2\bfj + q_3\bfk \mid q_\ell \in \Q \}
\end{equation*}
It is well known that the minimal polynomial of the element $q = q_0 + q_1\bfi + q_2\bfj + q_3\bfk \in B \setminus \Q$ is $\mu_q(X) = X^2 - 2q_0X + (q_0^2 + q_1^2 + q_2^2 + q_3^2)$, so $A'$ is the set
\begin{equation*}
A' = \{q_0 + q_1\bfi + q_2\bfj + q_3\bfk \in B \mid 2q_0, \; q_0^2 + q_1^2 + q_2^2 + q_3^2 \in \Z\}
\end{equation*} 

As with the previous example, by (\ref{Diagram}), it is enough to prove that $\Omega_{A'} = \Omega_A$. 

Let $q = q_0 + q_1\bfi + q_2\bfj + q_3\bfk \in A'$ and $N = q_0^2 + q_1^2 + q_2^2 + q_3^2 \in \Z$. Then, $2q_0 \in \Z$, so $q_0$ is either an integer or a half-integer. If $q_0 \in \Z$, then $q_1^2 + q_2^2 + q_3^2 = N - q_0^2 \in \Z$. It is known (see for instance \cite[Lem.\ B p.\ 46]{Serre}) that an integer which is the sum of three rational squares is a sum of three integer squares. Thus, there exist $a_1, a_2, a_3 \in \Z$ such that $a_1^2 + a_2^2 + a_3^2 = N - q_0^2$. Then, $q' = q_0 + a_1 \bfi + a_2 \bfj + a_3 \bfk$ is an element of $A$ such that $\Omega_{q'} = \Omega_{q}$.

If $q_0$ is a half-integer, then $q_0 = \tfrac{t}{2}$ for some odd $t \in \Z$. In this case, $q_1^2 + q_2^2 + q_3^2 = \frac{4N - t^2}{4} = \frac{u}{4}$, where $u \equiv 3$ mod 4. Clearing denominators, we get $(2q_1)^2 + (2q_2)^2 + (2q_3)^2 = u$. As before, there exist integers $a_1, a_2$, and $a_3$ such that $a_1^2 + a_2^2 + a_3^2 = u$. But since $u \equiv 3$ mod 4, each of the $a_\ell$ must be odd. So, $q' = (t + a_1 \bfi + a_2 \bfj + a_3 \bfk)/2 \in A$ is such that $\Omega_{q'} = \Omega_{q}$.

It follows that $\Omega_{A'} = \Omega_A$ and thus that $\Int_K(A, A') = \Int_K(A')$.
\end{Ex}

\begin{Ex}\label{Lipschitz quats}
In contrast to the last example, the Lipschitz quaternions $A_1 = \Z \oplus \Z \bfi \oplus \Z \bfj \oplus \Z \bfk$ (where we only allow $a_\ell \in \Z$) are not polynomially dense in $A_1'$. With $A$ as in Example \ref{Hurwitz quats}, we have $A_1\subset A$, and both rings have the same $B$, so $A_1'=A'$. Our proof is identical to Example \ref{Z[sqrt-3]}. Working mod 2, the only possible minimal polynomials for elements of $A_1 \setminus \Z$ are $X^2$ and $X^2+1$. It follows that $f(X) = \frac{x^2(x^2+1)}{2} \in \Int_K(A_1)$. Let $\alpha = \frac{1 + \bfi + \bfj + \bfk}{2} \in A'$. Then, the minimal polynomial of $\alpha$ is $X^2 - X + 1$ (note that this minimal polynomial is shared by $\theta = \frac{1 + \sqrt{-3}}{2}$ in Example \ref{Z[sqrt-3]}). Just as in Example \ref{Z[sqrt-3]}, $f(\alpha) = -\frac{1}{2}$, which is not in $A'$. Thus, $\Int_K(A_1) \not\subseteq \Int(A')$, so $A_1$ is not polynomially dense in $A_1'=A'$ by Proposition \ref{CC IV.4.9}.
\end{Ex}

\section{Further questions}
Here, we list more questions for further investigation.

\begin{Ques}
Under what conditions do we have equalities in (\ref{Diagram})? In particular, what are necessary and sufficient conditions on $A$ for $A$ to be polynomially dense in $A'$? In Examples \ref{Triangular mats} and \ref{Hurwitz quats}, we exploited the fact that if $\Omega_A = \Omega_{A'}$, then $\Int_K(A, A') = \Int_K(A')$. It is natural to ask whether the converse holds.
If $\Int_K(A, A') = \Int_K(A')$, does it follow that $\Omega_A = \Omega_{A'}$? In other words, if $A$ is polynomially dense in $A'$, then is 
$\Omega_A$ equal to $\Omega_{A'}$?
\end{Ques}

\begin{Ques}
By \cite[Proposition IV.4.1]{CaCh} it follows that $\Int(D)$ is integrally closed if and only if $D$ is integrally closed. By Theorem \ref{Int_K(A, A') is integrally closed} we know that if $A=A'$, then $\Int_K(A)$ is integrally closed. Do we have a converse? Namely, if $\Int_K(A)$ is integrally closed, can we deduce that $A=A'$?
\end{Ques}

\begin{Ques}
In our proof (Theorem \ref{Integral closure of Int(A)}) that $\Int_K(A, A')$ is the integral closure of $\Int_K(A)$, we needed the assumption that $D$ has finite residue rings. Is the theorem true without this assumption? In particular, is it true whenever $D$ is Noetherian? 
\end{Ques}

\begin{Ques}
When is $\Int_K(A, A') = \Int_K(\Omega_A, \Ln)$ a Pr\"ufer domain? When $D = \Z$, $\Int_\Q(A, A')$ is always Pr\"ufer by \cite[Cor.\ 4.7]{LopWer}. 
On the other hand, even when $A = D$ is a Pr\"{u}fer domain, $\Int(D)$ need not be Pr\"{u}fer (see \cite[Sec.\ IV.4]{CaCh}).
\end{Ques}

\begin{Ques}
In Remark \ref{Pullback remark}, we proved that $\Int_K(M_n(D))$ equals an intersection of pullbacks:
\begin{equation*}
\bigcap_{M \in M_n(D)} (D[X] + \mu_M(X) \cdot K[X]) = \Int_K(M_n(D))
\end{equation*}
Does such an equality hold for other algebras?
\end{Ques}

\subsection*{Acknowledgement}
The authors wish to thank the referee for his/her suggestions. 
The first author wishes to thank Daniel Smertnig for useful discussions during the preparation of this paper about integrality in non-commutative settings. The same author was supported by the Austrian Science Foundation (FWF), Project Number P23245-N18. 


\end{document}